\documentclass[a4paper, 12pt, reqno]{amsart}

\usepackage[colorinlistoftodos]{todonotes}
\usepackage{amsmath, amsthm, amssymb, amsfonts, fullpage,hyperref}
\usepackage{mathtools}
\usepackage{graphicx, subfigure, color,calc}
\usepackage{tikz}
\usetikzlibrary{calc}
\usepackage{pgfplots}
\pgfplotsset{compat=1.17}

\newcommand{\N}{\mathbb N}

\newcommand{\E}{\mathbb{E}}

\renewcommand{\E}{\mathbb E}
\renewcommand{\S}{\mathcal S}
\renewcommand{\d}{\mathrm d}

\theoremstyle{plain}
\newtheorem{thm}{Theorem}[section]
\newtheorem{cor}[thm]{Corollary}
\newtheorem{lemma}[thm]{Lemma}

\theoremstyle{remark}
\newtheorem{remark}[thm]{Remark}

\usepackage[normalem]{ulem}

\title{Asymptotic Normality of Centroids of Random Polygons}

\author{Thorsten Neuschel$^\dag$}
\address{$^\dag$School of Mathematical Sciences, Dublin City University, Ireland}
\email{thorsten.neuschel@dcu.ie}

\begin{document}
\maketitle
\begin{abstract}  We explore the asymptotic behavior of the centroids of random polygons constructed from regular polygons with vertices on the unit circle by extending the rays so that their lengths form a random permutation of the first \(n\) integers. Surprisingly, this question has connections to diverse mathematical contexts, including random matrix theory and discrete Fourier transforms. Through rigorous analysis, we establish that the sequence of the suitably rescaled centroids converges to a circularly-symmetric complex normal distribution with variance \(\frac{1}{12}\). This result is a manifestation of central limit behavior in a setting involving sums of heavily dependent random variables.

\end{abstract}
\section{Introduction and Statement of Results}
For a positive integer \(n \in \mathbb{N}\), let us consider the regular polygon with vertices on the unit circle at the roots of unity \(\zeta_n^{}, \zeta_n^2 , \ldots, \zeta_n^n\), where \(\zeta_n = e^{2\pi i /n}\). We extend each of the rays \([0,\zeta_n^{} ]\),  \([0,\zeta_n^{2} ]\), \ldots,  \([0,\zeta_n^{n} ]\) by fixing the point at the origin and assigning a random permutation of the lengths \(1,\ldots,n\) to the segments. We obtain a polygon in the complex plane with random vertices 
\[\sigma(1)\zeta_n^{}, \sigma(2)\zeta_n^2 , \ldots, \sigma(n)\zeta_n^n,\] 
where \(\sigma \) is a random permutation from the symmetric group \(\mathcal{S}_n\),  acting on the numbers \(1,2,\ldots,n\).

\begin{figure}[h]
    \centering
 \begin{minipage}{0.42\textwidth}
        \centering
        \begin{tikzpicture}[scale=0.8]
            \def\n{5}
            
            \foreach \k in {1,2,...,\n} {
                \coordinate (z\k) at ({cos(360/\n * (\k-1))},{sin(360/\n * (\k-1))});
            }

            \def\perm{3, 5, 1, 4, 2}
            
            \foreach \k [count=\i from 1] in \perm {
                \draw[thick] (0,0) -- ($\k*(z\i)$);
                \node at ($\k*(z\i)$) [above right] {\(\k\)};
            }

            \draw[thick] 
                ($3*(z1)$) -- ($5*(z2)$) -- ($1*(z3)$) -- ($4*(z4)$) -- ($2*(z5)$) -- cycle;
            
            \draw[dashed] (0,0) circle (1);
            
            \node at (0,0) [below left] {O};
        \end{tikzpicture}
        \caption{A polygon with vertices determined by a permutation of lengths and roots of unity on the unit circle (\(n=5\)).}
        \label{fig:polygon5}
    \end{minipage}
    \begin{minipage}{0.42\textwidth}
        \centering
        \begin{tikzpicture}[scale=0.8]
            \def\n{6}
            
            \foreach \k in {1,2,...,\n} {
                \coordinate (z\k) at ({cos(360/\n * (\k-1))},{sin(360/\n * (\k-1))});
            }

            \def\perm{2, 5, 6, 1, 4, 3}
            
            \foreach \k [count=\i from 1] in \perm {
                \draw[thick] (0,0) -- ($\k*(z\i)$);
                \node at ($\k*(z\i)$) [above right] {\(\k\)};
            }

            \draw[thick] 
                ($2*(z1)$) -- ($5*(z2)$) -- ($6*(z3)$) -- ($1*(z4)$) -- ($4*(z5)$) -- ($3*(z6)$) -- cycle;
            
            \draw[dashed] (0,0) circle (1);
            
            \node at (0,0) [below left] {O};
        \end{tikzpicture}
        \caption{A polygon with vertices determined by a permutation of lengths and roots of unity on the unit circle (\(n=6\)).}
        \label{fig:polygon6}
    \end{minipage}
   
\end{figure}


 In the following, for every \(n\in\mathbb{N}\), we choose \(\sigma = \sigma_n\) to be a uniformly distributed \(\mathcal{S}_n\)-valued random variable, defined on the probability space \((\Omega, \mathcal{A}, \mathbb{P})\). The polygon's \emph{centroid} \(C_n\) is given by
\begin{equation*}
C_n = C_n(\sigma) = \frac{1}{n}\sum_{k=1}^n \sigma(k) e^{2\pi i k /n},\quad n \geq 1.
\end{equation*}
Hence, \((C_n)_n\) is a sequence of discrete complex-valued random variables, and each \(C_n\) is the sum of \emph{stochastically dependent} random variables. It is straightforward to check that the variables \(C_n\) are centered in expectation
\[\E [C_n] = \int_{\S_n}C_n(\sigma) \d \sigma = \frac{1}{n}\sum_{k=1}^n \frac{1}{n!}\left(\sum_{\sigma\in\S_n} \sigma(k) \right) \zeta_n^k = \frac{n+1}{2n}\sum_{k=1}^n \zeta_n^k =0,\]
however, it seems very difficult to describe their distribution precisely for finite values of \(n\). This is a reason why asymptotic analysis is so central in probability theory, it provides clear and often powerful insights into the behavior of random processes, especially when finite-level details are complex or unwieldy. Investigating the probabilistic properties of the quantities \((C_n)_n\), such as their distribution or distance from the expected position (the origin), is also motivated by the fact that this problem naturally appears in other areas under different guises. For instance, let us consider mass points \(m_1,m_2,\ldots,m_n\) with \(m_k =k\) for \(1\leq k \leq n\), and distribute them over the unit circle at the roots of unity in the following way: given a permutation \(\sigma \in S_n\), we place the mass \(m_{\sigma(k)}\) at the point \(e^{2\pi i k /n}\) for each \(k=1,\ldots,n\). The \emph{center of mass} or \emph{barycenter} then is given by
\begin{equation*}B_n = \frac{2}{n(n+1)} \sum_{k=1}^n \sigma(k) e^{2\pi i k /n} = \frac{2}{n+1}C_n,
\end{equation*}
where the normalization for the \(B_n\) comes from the fact that we have 
\[\sum_{k=1}^n \sigma(k) = \frac{n(n+1)}{2}.\]
It is a natural question to ask about the location of the center of mass for a generic permutation \(\sigma\).
We can also express \(C_n\) in terms of unitary \emph{random matrices} by
\[C_n = \frac{1}{n} \left(\zeta_n , \zeta_n^2, \ldots, \zeta_n^n \right)P_n(\sigma) \begin{pmatrix} 1\\ 2\\ \vdots \\ n\end{pmatrix},\]
where \(P_n(\sigma)\) is the random permutation matrix \(\left[\delta_{\sigma(j),k}\right]_{j,k=1}^n\), and \(\sigma \) is drawn uniformly from the symmetric group \(\mathcal{S}_n\). From this perspective, asking about the asymptotic behavior of \(C_n\) becomes a question from high-dimensional random matrix theory. Alternatively, it is noteworthy, without going into any details, that the centroids \(C_n\) can also be interpreted in terms of the \emph{discrete Fourier transform} of the random sequence \(\sigma(1),\ldots, \sigma(n)\), and it is certainly an interesting question how Fourier transforms with random input behave. A further historical origin of the interest in the behavior of \(C_n\), depending on the permutation \(\sigma\), stems from a problem proposed for the International Mathematical Olympiad in
Paris in the year 1983, the story of which is described in \cite{Woeg}: it turns out that the equation 
\begin{equation}\label{zero}C_n =  \frac{1}{n}\sum_{k=1}^n \sigma(k) e^{2\pi i k /n} = 0
\end{equation}
has a solution \(\sigma\in\mathcal{S}_n\) if and only if \(n\) has two distinct prime factors. For a recent generalisation of this result we refer to \cite{Abel}. In this light, it is a natural question to ask what can be said about the location of \(C_n\) for a generic permutation \(\sigma\), and how close is it to being a solution of equation \eqref{zero}. It is the goal of the present work to answer these questions rigorously for large values of \(n\) in the elementary setting for \(C_n\), while bearing in mind that there are many valid ways to develop and generalise these questions further in the aforementioned areas. The main result here is the following.

\begin{thm}\label{MainThm} Let \(X_n =\frac{1}{\sqrt{n}} \Re{[C_n]}\) and \(Y_n = \frac{1}{\sqrt{n}} \Im{[C_n]}\) for \(n\in\mathbb{N}\). Then the distribution of \((X_n,Y_n)^T\) converges weakly, as \(n\to\infty\), to a two-dimensional  Gaussian distribution with mean \(\mu = (0,0)^T\) and covariance matrix 
\[\Gamma = \begin{pmatrix} \frac{1}{24} & 0 \\ 0& \frac{1}{24} \end{pmatrix}.\] 
In particular, the sequence of rescaled centroids \(( Z_n)_n\) with \(Z_n = \frac{1}{\sqrt{n}}C_n\) converges weakly, as \(n\to\infty\), to a circularly-symmetric central complex normal distribution with variance \(\frac{1}{12}\), see Figure \ref{fig:3d_plot}.
\end{thm}

\begin{remark} \begin{enumerate}

 \item The statement of Theorem \ref{MainThm} can be interpreted as a \emph{central limit theorem} for the specific type of sums of (heavily) dependent random variables considered here.
\item From the perspective of \emph{universality} as studied in probability and random matrix theory, the statement confirms and reinforces the universal nature of the normal distribution by adding a further example to the list of sums of random variables that behave asymptotically normal.
\end{enumerate}
\end{remark}

\begin{figure}[h]
    \centering
    \includegraphics[scale=0.7]{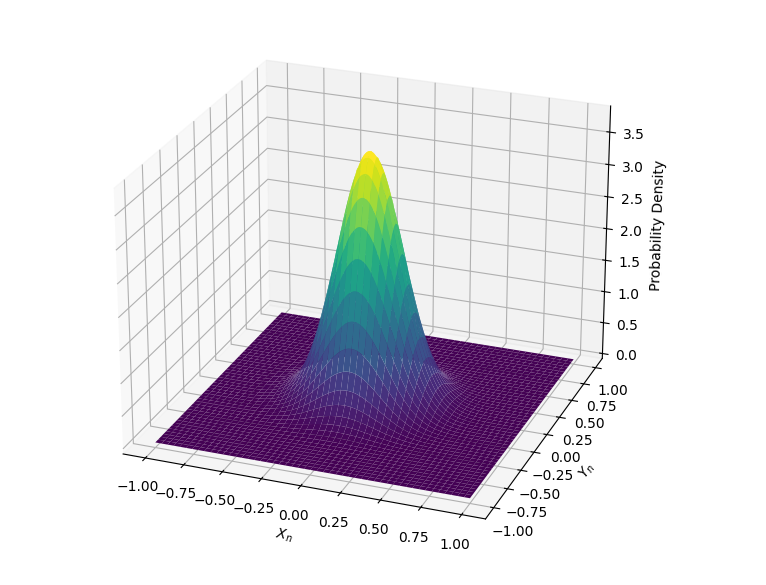}
    \caption{Density of the limiting distribution of the rescaled centroids on the complex plane.}
    \label{fig:3d_plot}
\end{figure}

From Theorem \ref{MainThm}, the following corollary immediately follows.
\begin{cor} The distance from the rescaled centroid \(Z_n\) to the origin is asymptotically, for large \(n\), distributed according to the Rayleigh distribution. More precisely, the distribution of \(\vert Z_n \vert \) converges weakly, as \(n\to\infty\), to the Rayleigh distribution with parameter \(\sigma^2 = \frac{1}{24}\) with density, see Figure \ref{fig:rayleigh_density},
\[f_{Ray}(x)=\frac{x}{\sigma^2}e^{-\frac{x^2}{2\sigma^2}}= 24x e^{-12 x^2}, \quad x \geq 0.\]
\end{cor}

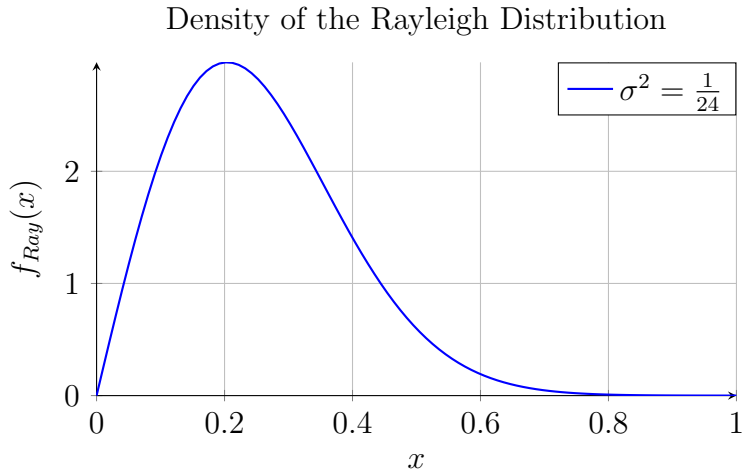
\begin{figure}[h]
    \centering
    \begin{tikzpicture}
        \begin{axis}[
            domain=0:1, 
            samples=100, 
            xlabel={$x$}, 
            ylabel={$f_{Ray}(x)$}, 
            axis lines=left, 
            width=10cm, 
            height=6cm, 
            grid=major, 
            title={Density of the Rayleigh Distribution}, 
            legend style={at={(1,1)}, anchor=north east} 
        ]
        \addplot[
            thick, 
            blue, 
        ] {24*x*exp(-12*x^2)}; 
        \addlegendentry{$\sigma^2 = \frac{1}{24}$}
        
        \end{axis}
    \end{tikzpicture}
    \caption{Density of the limiting distribution of the distance of the rescaled centroids from the origin.}
    \label{fig:rayleigh_density}
\end{figure}

\section{Proofs}

In this section we will give the proof of Theorem \ref{MainThm}. We focus on the real parts \(X_n =\frac{1}{\sqrt{n}} \Re{[C_n]}\), because the analysis of the imaginary parts \(Y_n\) works analogously and requires only minor modifications. Expanding the characteristic function of \(X_n\) in the usual manner, we obtain
\begin{equation*}\varphi_{X_n}(t) = \E \left[ e^{it X_n}\right] = \int_{\S_n} \sum_{m=0}^{\infty} \frac{i^m}{m!}t^m X_n^m \d \sigma =\sum_{m=0}^{\infty} \frac{i^m}{m!} t^m \int_{\S_n} X_n^m \d\sigma, 
\end{equation*}
with the \(m\)-th moment of \(X_n\) given by
\[ \int_{\S_n} X_n^m \d\sigma = \frac{1}{n!} \sum_{\sigma\in\S_n} \left( \sum_{j=1}^n \frac{\sigma(j)}{n^{3/2}} \cos \left(\frac{2\pi j}{n}\right)\right)^m.\]
By  Levy's continuity theorem, e.g. see Theorem 26.3 in \cite{Bill}, the statement of Theorem \ref{MainThm} regarding \(X_n\) is equivalent to the convergence  
\begin{equation}\label{l0}\lim_{n\to\infty} \varphi_{X_n}(t) = e^{-\frac{1}{48} t^2},\quad t\in \mathbb{R}, \end{equation}
which means that the characteristic function of \(X_n\) converges to the characteristic function of a normal distribution with mean zero and variance \(\sigma^2 = 1/24\). As any normal distribution is uniquely determined by its moments, it is sufficient for \eqref{l0} to prove the convergence of the moments, see e.g., Theorem 30.2 \cite{Bill}. That constitutes the core of this work and is established in Lemma \ref{lemma1}.

\begin{lemma} \label{lemma1} For \(m\in\N\) we have
\begin{equation} \label{l1}\lim_{n\to\infty}  \int_{\S_n} X_n^{2m} \d\sigma  =\frac{(2m)!}{m!}\frac{1}{(48)^m}\end{equation}
and
\begin{equation} \label{l2} \lim_{n\to\infty}  \int_{\S_n} X_n^{2m-1} \d\sigma  =0. \end{equation}
\end{lemma}
We will prove Lemma \ref{lemma1} below. The analogous result also holds true for the imaginary parts \(Y_n\). Moreover, we can easily find that the real and imaginary parts are \emph{uncorrelated} for \(n>1\):
\begin{align}\nonumber \mathrm{Cov}\left(X_n,Y_n\right) &= \E\left[X_n Y_n\right] =  \frac{1}{n}\E\left[\Bigg(\frac{1}{n}\sum_{j=1}^n  \sigma(j)\cos\left(\frac{2\pi j}{n}\right)\Bigg)\Bigg( \frac{1}{n}\sum_{k=1}^n \sigma(k) \sin\left(\frac{2\pi k}{n}\right)\Bigg)\right]\\
&= \frac{1}{n^3} \sum_{j,k=1}^n \frac{1}{n!}\Bigg(\sum_{\sigma\in\S} \sigma(j) \sigma(k)\Bigg) \cos\left(\frac{2\pi j}{n}\right)\cos\left(\frac{2\pi k}{n}\right)\label{l-1}
\end{align}
An elementary computation shows
\[\sum_{j=1}^n \cos \left( \frac{2\pi j}{n}\right) = \sum_{k=1}^n \sin \left( \frac{2\pi k}{n}\right)= \sum_{j=1}^n \cos \left( \frac{2\pi j}{n}\right)\sin \left( \frac{2\pi j}{n}\right)=0,\]
as well as
\[\frac{1}{n!}\sum_{\sigma\in\S} \sigma(j) \sigma(k) = a_n \delta_{j,k} + b_n (1-\delta_{j,k}),\]
where \(\delta_{j,k}\) is Kronecker's delta and the expressions \(a_n, b_n\) are given by
\[a_n = \frac{(n+1)(2n+1)}{6} \quad \text{and}\quad b_n = \frac{1}{n(n-1)}\left(\frac{n^2(n+1)^2}{4} - \frac{n(n+1)(2n+1)}{6}\right).\]
Using this in \eqref{l-1} gives
\begin{align}\nonumber \mathrm{Cov}\left(X_n,Y_n\right) &= \frac{a_n}{n^3} \sum_{j=1}^n  \cos \left( \frac{2\pi j}{n}\right)\sin \left( \frac{2\pi j}{n}\right)+  \frac{b_n}{n^3} \sum_{j\neq k} \cos \left( \frac{2\pi j}{n}\right) \sin \left( \frac{2\pi k}{n}\right) \\
&=  \frac{b_n}{n^3} \Bigg( \sum_{j=1}^n \cos \left( \frac{2\pi j}{n}\right)  \sum_{k=1}^n \sin \left( \frac{2\pi k}{n}\right) -   \sum_{j=1}^n  \cos \left( \frac{2\pi j}{n}\right)\sin \left( \frac{2\pi j}{n}\right) \Bigg)\nonumber\\
&=0. \nonumber
\end{align}

The rest of the paper is devoted to the proof of Lemma \ref{lemma1}.
\begin{proof}[Proof of Lemma \ref{lemma1}] We start with the even moments in \eqref{l1}, and wherever we need an additional lemma, we will treat it along the way. In the following we always consider \(n\geq 2m\). We have

\begin{align*}  \int_{\S_n} X_n^{2m} \d\sigma &= \frac{1}{n!} \sum_{\sigma\in\S_n} \left( \sum_{j=1}^n \frac{\sigma(j)}{n^{3/2}} \cos \left(\frac{2\pi j}{n}\right)\right)^{2m}\\
&= \frac{1}{n^{3m}} \frac{1}{n!} \sum_{\sigma\in\S_n} \sum_{\substack{\alpha_1 , \ldots, \alpha_n \geq 0 \\ \alpha_1 + \ldots + \alpha_n = 2m }} \frac{(2m)!}{\alpha_1! \cdots \alpha_n!} \sigma(1)^{\alpha_1} \cdots \sigma(n)^{\alpha_n} \cos \left(\frac{2\pi }{n}\right)^{\alpha_1}\cdots  \cos \left(\frac{2\pi n }{n}\right)^{\alpha_n},
\end{align*}
where we use the multinomial theorem to expand the sum \( \left( \sum_{j=1}^n \frac{\sigma(j)}{n^{3/2}} \cos \left(\frac{2\pi j}{n}\right)\right)^{2m}\). As the summation of the the inner sum is over the set of all tuples \((\alpha_1, \ldots, \alpha_n)\) of integers \(\alpha_1 , \ldots, \alpha_n \geq 0\) with \(  \alpha_1 + \ldots + \alpha_n = 2m\), we can partition this set into the sets of all tuples with exactly \(j\) \emph{positive} integers for \(j=1,\ldots, 2m\), occuring at the places \(1\leq i_1 < i_2 < \ldots < i_j \leq n \). This yields
\begin{align*}  \int_{\S_n} X_n^{2m} \d\sigma &=  \frac{1}{n^{3m}} \frac{1}{n!} \sum_{\sigma\in\S_n}\sum_{j=1}^{2m} \sum_{1\leq i_1 < \ldots < i_j \leq n} \\
& \quad\quad \sum_{\substack{\alpha_1 , \ldots, \alpha_j \in \{1,\dots, 2m\} \\ \alpha_1 + \ldots + \alpha_j = 2m }}   \frac{(2m)!}{\alpha_1! \cdots \alpha_j!} \sigma(i_1)^{\alpha_1} \cdots \sigma(i_j)^{\alpha_j} \cos \left(\frac{2\pi i_1 }{n}\right)^{\alpha_1}\cdots  \cos \left(\frac{2\pi i_j }{n}\right)^{\alpha_j} \\
&= \frac{1}{n^{3m}} \sum_{j=1}^{2m} \sum_{\substack{\alpha_1 , \ldots, \alpha_j \in \{1,\dots, 2m\} \\ \alpha_1 + \ldots + \alpha_j = 2m }} \frac{(2m)!}{\alpha_1! \cdots \alpha_j!}  \sum_{1\leq i_1 < \ldots < i_j \leq n}\cos \left(\frac{2\pi i_1 }{n}\right)^{\alpha_1}\cdots  \cos \left(\frac{2\pi i_j }{n}\right)^{\alpha_j}\\
&   \quad\quad \frac{1}{n!} \sum_{\sigma\in\S_n} \sigma(i_1)^{\alpha_1} \cdots \sigma(i_j)^{\alpha_j}.
\end{align*}
Let us rewrite the innermost sum for given \(\alpha_1 , \ldots, \alpha_j \in \{1,\dots, 2m\} \) with \( \alpha_1 + \ldots + \alpha_j = 2m\) and  \(1\leq i_1 <\ldots< i_j \leq n\) in the following way
\begin{align*}  \frac{1}{n!} \sum_{\sigma\in\S_n} \sigma(i_1)^{\alpha_1} \cdots \sigma(i_j)^{\alpha_j} =& \frac{1}{n!} \sum_{\substack{k_1 \neq \ldots \neq k_j \\  k_1, \ldots, k_j \in \{1,\ldots, n\} \\ } }  \sum_{\substack{\sigma\in\S_n \\ \sigma(i_1)=k_1, \ldots, \sigma(i_j)= k_j} }  \sigma(i_1)^{\alpha_1} \cdots \sigma(i_j)^{\alpha_j} \\
=&  \frac{1}{n!} \sum_{\substack{k_1 \neq \ldots \neq k_j \\  k_1, \ldots, k_j \in \{1,\ldots, n\} \\ } } k_1^{\alpha_1} \cdots k_j^{\alpha_j}  \sum_{\substack{\sigma\in\S_n \\ \sigma(i_1)=k_1, \ldots, \sigma(i_j)= k_j} }1\\
=& \frac{(n-j)!}{n!} \sum_{\substack{k_1 \neq \ldots \neq k_j \\  k_1, \ldots, k_j \in \{1,\ldots, n\}  } } k_1^{\alpha_1} \cdots k_j^{\alpha_j},
\end{align*}
where the summation over \(k_1 \neq \ldots \neq k_j\) with \(  k_1, \ldots, k_j \in \{1,\ldots, n\}\) means summation over the set of all tuples \((k_1,\ldots, k_j) \in \{1,\ldots, n\}^j \) with pairwise distinct entries. In particular, it follows that the expression \( \frac{1}{n!} \sum_{\sigma\in\S_n} \sigma(i_1)^{\alpha_1} \cdots \sigma(i_j)^{\alpha_j}\) is independent of the concrete choice of indices \(1\leq i_1 <\ldots< i_j \leq n\). Hence, we obtain

\begin{align*}  \int_{\S_n} X_n^{2m} \d\sigma &=  \frac{1}{n^{3m}}\sum_{j=1}^{2m} \frac{(n-j)!}{n!}  \sum_{\substack{\alpha_1 , \ldots, \alpha_j \in \{1,\dots, 2m\} \\ \alpha_1 + \ldots + \alpha_j = 2m }} \frac{(2m)!}{\alpha_1! \cdots \alpha_j!} \Bigg(\sum_{\substack{k_1 \neq \ldots \neq k_j \\  k_1, \ldots, k_j \in \{1,\ldots, n\}  } } k_1^{\alpha_1} \cdots k_j^{\alpha_j}\Bigg)\\
&\quad \quad  \sum_{1\leq i_1 < \ldots < i_j \leq n}\cos \left(\frac{2\pi i_1 }{n}\right)^{\alpha_1}\cdots  \cos \left(\frac{2\pi i_j }{n}\right)^{\alpha_j}.
\end{align*}
In order to proceed, for each fixed \(j=1,\ldots, 2m\), we introduce an equivalence relation \(\mathcal{R}_j\) on the set of tuples \((\alpha_1, \ldots, \alpha_j)\) with  \(\alpha_1 , \ldots, \alpha_j \in \{1,\dots, 2m\} \) and \( \alpha_1 + \ldots + \alpha_j = 2m\): we write \((\alpha_1, \ldots, \alpha_j) \sim (\beta_1, \ldots, \beta_j)\) if there is a permutation \(\mu\in\S_j\) such that \[(\alpha_1, \ldots, \alpha_j) =(\beta_{\mu(1)}, \ldots, \beta_{\mu(j)}). \]
Moreover, for each equivalence class we choose one of its contained tuples \((\alpha_1, \ldots, \alpha_j)\), and, as usual, we denote this equivalence class by \([(\alpha_1, \ldots, \alpha_j)]\). This gives
\begin{align*}  \int_{\S_n} X_n^{2m} \d\sigma &=  \frac{1}{n^{3m}}\sum_{j=1}^{2m} \frac{(n-j)!}{n!} \sum_{[(\alpha_1, \ldots, \alpha_j)] \in \mathcal{R}_j} \sum_{(\beta_1, \ldots, \beta_j) \sim (\alpha_1, \ldots, \alpha_j)}  \frac{(2m)!}{\beta_1! \cdots \beta_j!}\\
&\quad\quad   \Bigg(\sum_{\substack{k_1 \neq \ldots \neq k_j \\  k_1, \ldots, k_j \in \{1,\ldots, n\}  } } k_1^{\beta_1} \cdots k_j^{\beta_j}  \Bigg) \sum_{1\leq i_1 < \ldots < i_j \leq n}\cos \left(\frac{2\pi i_1 }{n}\right)^{\beta_1}\cdots  \cos \left(\frac{2\pi i_j }{n}\right)^{\beta_j}. 
\end{align*}
We observe for \((\alpha_1, \ldots, \alpha_j) \sim (\beta_1, \ldots, \beta_j)\)
\[\frac{(2m)!}{\beta_1! \cdots \beta_j!} = \frac{(2m)!}{\alpha_1! \cdots \alpha_j!} \quad \text{and} \quad  \sum_{\substack{k_1 \neq \ldots \neq k_j \\  k_1, \ldots, k_j \in \{1,\ldots, n\}  } } k_1^{\beta_1} \cdots k_j^{\beta_j} =  \sum_{\substack{k_1 \neq \ldots \neq k_j \\  k_1, \ldots, k_j \in \{1,\ldots, n\}  } } k_1^{\alpha_1} \cdots k_j^{\alpha_j}, \]
so these quantities only depend on the equivalence class of the tuples, which yields
\begin{align}  \int_{\S_n} X_n^{2m} \d\sigma &=  \frac{1}{n^{3m}}\sum_{j=1}^{2m} \frac{(n-j)!}{n!} \sum_{[(\alpha_1, \ldots, \alpha_j)] \in \mathcal{R}_j} \frac{(2m)!}{\alpha_1! \cdots \alpha_j!}  \Bigg(\sum_{\substack{k_1 \neq \ldots \neq k_j \\  k_1, \ldots, k_j \in \{1,\ldots, n\}  } } k_1^{\alpha_1} \cdots k_j^{\alpha_j} \Bigg)\\
&\quad\quad  \sum_{(\beta_1, \ldots, \beta_j) \sim (\alpha_1, \ldots, \alpha_j)}  \sum_{1\leq i_1 < \ldots < i_j \leq n}\cos \left(\frac{2\pi i_1 }{n}\right)^{\beta_1}\cdots  \cos \left(\frac{2\pi i_j }{n}\right)^{\beta_j}. \label{l3}
\end{align}
The reason for the introduction of the equivalence classes \(\mathcal{R}_j\) above is that the expression
\[ \sum_{(\beta_1, \ldots, \beta_j) \sim (\alpha_1, \ldots, \alpha_j)} \cos \left(\frac{2\pi i_1 }{n}\right)^{\beta_1}\cdots  \cos \left(\frac{2\pi i_j }{n}\right)^{\beta_j}\]
is invariant under permutations of the indices \(1\leq i_1 <\cdots < i_j \leq n\), so we have
\begin{align*}& \sum_{(\beta_1, \ldots, \beta_j) \sim (\alpha_1, \ldots, \alpha_j)}  \sum_{1\leq i_1 < \ldots < i_j \leq n}\cos \left(\frac{2\pi i_1 }{n}\right)^{\beta_1}\cdots  \cos \left(\frac{2\pi i_j }{n}\right)^{\beta_j}\\
&=  \sum_{1\leq i_1 < \ldots < i_j \leq n} \sum_{(\beta_1, \ldots, \beta_j) \sim (\alpha_1, \ldots, \alpha_j)} \cos \left(\frac{2\pi i_1 }{n}\right)^{\beta_1}\cdots  \cos \left(\frac{2\pi i_j }{n}\right)^{\beta_j}\\
&= \frac{1}{j!} \sum_{\substack{i_1 \neq \ldots \neq i_j \\  i_1, \ldots, i_j \in \{1,\ldots, n\} } }\sum_{(\beta_1, \ldots, \beta_j) \sim (\alpha_1, \ldots, \alpha_j)} \cos \left(\frac{2\pi i_1 }{n}\right)^{\beta_1}\cdots  \cos \left(\frac{2\pi i_j }{n}\right)^{\beta_j}\\
&= \frac{1}{j!} \sum_{(\beta_1, \ldots, \beta_j) \sim (\alpha_1, \ldots, \alpha_j)}  \sum_{\substack{i_1 \neq \ldots \neq i_j \\  i_1, \ldots, i_j \in \{1,\ldots, n\} } } \cos \left(\frac{2\pi i_1 }{n}\right)^{\beta_1}\cdots  \cos \left(\frac{2\pi i_j }{n}\right)^{\beta_j}.
\end{align*}
Subsituting this into \eqref{l3} gives us
\begin{align}\nonumber  \int_{\S_n} X_n^{2m} \d\sigma &=  \frac{1}{n^{3m}}\sum_{j=1}^{2m} \frac{(n-j)!}{n!} \sum_{[(\alpha_1, \ldots, \alpha_j)] \in \mathcal{R}_j} \frac{(2m)!}{\alpha_1! \cdots \alpha_j!}  \Bigg(\sum_{\substack{k_1 \neq \ldots \neq k_j \\  k_1, \ldots, k_j \in \{1,\ldots, n\}  } } k_1^{\alpha_1} \cdots k_j^{\alpha_j} \Bigg)\\ \nonumber
&\quad\quad   \frac{1}{j!} \sum_{(\beta_1, \ldots, \beta_j) \sim (\alpha_1, \ldots, \alpha_j)}  \sum_{\substack{i_1 \neq \ldots \neq i_j \\  i_1, \ldots, i_j \in \{1,\ldots, n\} } } \cos \left(\frac{2\pi i_1 }{n}\right)^{\beta_1}\cdots  \cos \left(\frac{2\pi i_j }{n}\right)^{\beta_j}\\ \nonumber
&=  \frac{1}{n^{3m}}\sum_{j=1}^{2m}  \frac{1}{j!}  \frac{(n-j)!}{n!}  \sum_{[(\alpha_1, \ldots, \alpha_j)] \in \mathcal{R}_j}\sum_{(\beta_1, \ldots, \beta_j) \sim (\alpha_1, \ldots, \alpha_j)} \frac{(2m)!}{\beta_1! \cdots \beta_j!}\\ \nonumber
&\quad \Bigg(\sum_{\substack{k_1 \neq \ldots \neq k_j \\  k_1, \ldots, k_j \in \{1,\ldots, n\}  } } k_1^{\beta_1} \cdots k_j^{\beta_j} \Bigg) \Bigg( \sum_{\substack{i_1 \neq \ldots \neq i_j \\  i_1, \ldots, i_j \in \{1,\ldots, n\} } } \cos \left(\frac{2\pi i_1 }{n}\right)^{\beta_1}\cdots  \cos \left(\frac{2\pi i_j }{n}\right)^{\beta_j}\Bigg)\\ \nonumber
&=  \frac{1}{n^{3m}}\sum_{j=1}^{2m}  \frac{1}{j!}  \frac{(n-j)!}{n!}   \sum_{\substack{\alpha_1 , \ldots, \alpha_j \in \{1,\dots, 2m\} \\ \alpha_1 + \ldots + \alpha_j = 2m }} \frac{(2m)!}{\alpha_1! \cdots \alpha_j!}  \Bigg(\sum_{\substack{k_1 \neq \ldots \neq k_j \\  k_1, \ldots, k_j \in \{1,\ldots, n\}  } } k_1^{\alpha_1} \cdots k_j^{\alpha_j} \Bigg) \\
&\quad\quad  \sum_{\substack{i_1 \neq \ldots \neq i_j \\  i_1, \ldots, i_j \in \{1,\ldots, n\} } } \cos \left(\frac{2\pi i_1 }{n}\right)^{\alpha_1}\cdots  \cos \left(\frac{2\pi i_j }{n}\right)^{\alpha_j}. \label{l4}
\end{align}

Before proceeding with the proof, we need to deal with the asymptotic behavior of the multiple power sums appearing in \eqref{l4}.

\begin{lemma} \label{lemma2} For fixed integers \(j\geq 1\), \(\alpha_1, \ldots, \alpha_j \geq 1\), we have
\[\sum_{\substack{k_1 \neq \ldots \neq k_j \\  k_1, \ldots, k_j \in \{1,\ldots, n\}  } } k_1^{\alpha_1} \cdots k_j^{\alpha_j} = \frac{n^{\alpha_1 + \ldots + \alpha_j +j }}{(\alpha_1 +1)\cdots(\alpha_j +1)} \left(1+\mathcal{O}\left(\frac{1}{n}\right)\right),\quad n\to\infty, \]
where the constant implied by the \(\mathcal{O}\)-term can be chosen to be independent of \(n\), but it depends on \(j\) and \(\alpha_1, \ldots, \alpha_j\).
\end{lemma}
\begin{proof}[Proof of Lemma \ref{lemma2}] The case \(j=1, \alpha_1 = \alpha \geq 1\) is a consequence of the well-know Faulhaber's formula \cite{Knuth}
\[\sum_{k=1}^n k^\alpha = \frac{n^{\alpha+1}}{\alpha +1} \left(1+ \sum_{k=1}^\alpha \binom{\alpha +1}{k}  \frac{B_k}{n^k}\right) = \frac{n^{\alpha+1}}{\alpha +1} \left(1+\mathcal{O}\left(\frac{1}{n}\right)\right),\]
as \(n\to\infty\), where \(B_k\) are the Bernoulli numbers with the convention \(B_1 = \frac{1}{2}\). From this we immediately obtain for \(j\geq 1\)
\begin{align}\nonumber\sum_{k_1, \ldots, k_j \in \{1,\ldots,n\}} k_1^{\alpha_1} \cdots k_j^{\alpha_j} &= \prod_{\nu=1}^j \sum_{k=1}^n k^{\alpha_{\nu}}  =\prod_{\nu=1}^j \frac{n^{\alpha_\nu +1}}{\alpha_\nu +1} \left(1+\mathcal{O}\left(\frac{1}{n}\right)\right)\\
&= \frac{n^{\alpha_1 + \ldots + \alpha_j +j }}{(\alpha_1 +1)\cdots(\alpha_j +1)} \left(1+\mathcal{O}\left(\frac{1}{n}\right)\right),\quad n\to\infty.\label{l5}
\end{align}
Moreover, we observe
\[\Bigg\vert \sum_{\substack{k_1 \neq \ldots \neq k_j \\  k_1, \ldots, k_j \in \{1,\ldots, n\}  } } k_1^{\alpha_1} \cdots k_j^{\alpha_j} - \sum_{k_1, \ldots, k_j \in \{1,\ldots,n\}} k_1^{\alpha_1} \cdots k_j^{\alpha_j} \Bigg\vert = \sum_{I}k_1^{\alpha_1} \cdots k_j^{\alpha_j}, \]
where the summation is carried out over the set 
\[I = \left\{(k_1,\ldots, k_j)\in \{1, \ldots, n\}^j ~\vert ~  \exists  i_1 < i_2 ~\text{with}~ k_{i_1} = k_{i_2} \right\} .\]
We can estimate the latter sum by
\begin{align*}\sum_{I}k_1^{\alpha_1} \cdots k_j^{\alpha_j} \leq & \sum_{1\leq i_1 < i_2 \leq j}  \sum_{\substack{  k_1, \ldots, k_j \in \{1,\ldots, n\}  \\  k_{i_1} = k_{i_2} } } k_1^{\alpha_1} \cdots k_j^{\alpha_j}\\
=&   \sum_{1\leq i_1 < i_2 \leq j} \Bigg(\sum_{k_{i_1}=1}^n k_{i_i}^{\alpha_{i_1}+\alpha_{i_i}} \Bigg)  \prod_{\substack{\nu = 1\\ \nu \neq i_1, i_2}}^j \Bigg( \sum_{k_\nu =1}^n k_\nu^{\alpha_\nu} \Bigg).
\end{align*}
From the case \(j=1\) we know
\[ \sum_{k_{i_1}=1}^n k_{i_i}^{\alpha_{i_1}+\alpha_{i_2}} = \frac{n^{\alpha_{i_1}+ \alpha_{i_2} +1}}{\alpha_{i_1}+ \alpha_{i_2} +1}\left(1+\mathcal{O}\left(\frac{1}{n}\right)\right) \quad \text{and}\quad  \sum_{k_\nu =1}^n k_\nu^{\alpha_\nu} =\frac{n^{\alpha_{\nu}+1}}{\alpha_\nu +1}\left(1+\mathcal{O}\left(\frac{1}{n}\right)\right), \]
as \(n\to \infty\), which amounts to
\begin{equation}\label{l6} \sum_{I}k_1^{\alpha_1} \cdots k_j^{\alpha_j} = \mathcal{O}\left( n^{\alpha_1 + \cdots + \alpha_j + j-1}\right), 
\end{equation}
as \(n\to \infty\). Combining \eqref{l5} and \eqref{l6} gives 
\begin{align*}\sum_{\substack{k_1 \neq \ldots \neq k_j \\  k_1, \ldots, k_j \in \{1,\ldots, n\}  } } k_1^{\alpha_1} \cdots k_j^{\alpha_j} =&  \frac{n^{\alpha_1 + \ldots + \alpha_j +j }}{(\alpha_1 +1)\cdots(\alpha_j +1)} \left(1+\mathcal{O}\left(\frac{1}{n}\right)\right) + \mathcal{O}\left( n^{\alpha_1 + \cdots + \alpha_j + j-1}\right)\\
=&  \frac{n^{\alpha_1 + \ldots + \alpha_j +j }}{(\alpha_1 +1)\cdots(\alpha_j +1)} \left(1+\mathcal{O}\left(\frac{1}{n}\right)\right),\quad \text{as}\quad n \to \infty. 
  \end{align*}
\end{proof}
\bigskip
Applying Lemma \ref{lemma2} to expression \eqref{l4} leads to
\begin{align}\nonumber  \int_{\S_n} X_n^{2m} \d\sigma &=\sum_{j=1}^{2m}  \frac{1}{j!}    \sum_{\substack{\alpha_1 , \ldots, \alpha_j \in \{1,\dots, 2m\} \\ \alpha_1 + \ldots + \alpha_j = 2m }} \frac{(2m)!}{(\alpha_1 +1)! \cdots (\alpha_j+1)!} \frac{(n-j)!~ n^j}{n!} \left( 1+ \mathcal{O}\left(\frac{1}{n}\right)\right) \\
&\quad\quad \frac{1}{n^m} \sum_{\substack{i_1 \neq \ldots \neq i_j \\  i_1, \ldots, i_j \in \{1,\ldots, n\} } } \cos \left(\frac{2\pi i_1 }{n}\right)^{\alpha_1}\cdots  \cos \left(\frac{2\pi i_j }{n}\right)^{\alpha_j}, \quad n\to\infty. \label{l7}
\end{align}
We let the expression \(\left( 1+ \mathcal{O}\left(\frac{1}{n}\right)\right)\) absorb the fraction \(\frac{(n-j)!~ n^j}{n!}\), and use the trivial bound 
\[ \Bigg\vert \sum_{\substack{i_1 \neq \ldots \neq i_j \\  i_1, \ldots, i_j \in \{1,\ldots, n\} } } \cos \left(\frac{2\pi i_1 }{n}\right)^{\alpha_1}\cdots  \cos \left(\frac{2\pi i_j }{n}\right)^{\alpha_j} \Bigg\vert \leq n^j\]
in order to conclude that, for each \(j\in\{1,\ldots, m-1\}\), the whole expression
\begin{align*}  \frac{1}{j!} &   \sum_{\substack{\alpha_1 , \ldots, \alpha_j \in \{1,\dots, 2m\} \\ \alpha_1 + \ldots + \alpha_j = 2m }} \frac{(2m)!}{(\alpha_1 +1)! \cdots (\alpha_j+1)!} \frac{(n-j)!~ n^j}{n!} \left( 1+ \mathcal{O}\left(\frac{1}{n}\right)\right) \\
&\quad\quad \frac{1}{n^m} \sum_{\substack{i_1 \neq \ldots \neq i_j \\  i_1, \ldots, i_j \in \{1,\ldots, n\} } } \cos \left(\frac{2\pi i_1 }{n}\right)^{\alpha_1}\cdots  \cos \left(\frac{2\pi i_j }{n}\right)^{\alpha_j}
\end{align*}
is of order \( \mathcal{O}\left(\frac{1}{n}\right)\), as \(n\to\infty\). Hence, in \eqref{l7} we only have to deal with summands indexed by \(j \in \{m,\ldots, 2m\}\), which means
\begin{align}\nonumber  \int_{\S_n} X_n^{2m} \d\sigma &=\sum_{j=m}^{2m}  \frac{1}{j!}    \sum_{\substack{\alpha_1 , \ldots, \alpha_j \in \{1,\dots, 2m\} \\ \alpha_1 + \ldots + \alpha_j = 2m }} \frac{(2m)!}{(\alpha_1 +1)! \cdots (\alpha_j+1)!}\left( 1+ \mathcal{O}\left(\frac{1}{n}\right)\right) \\
&\quad\quad \frac{1}{n^m} \sum_{\substack{i_1 \neq \ldots \neq i_j \\  i_1, \ldots, i_j \in \{1,\ldots, n\} } } \cos \left(\frac{2\pi i_1 }{n}\right)^{\alpha_1}\cdots  \cos \left(\frac{2\pi i_j }{n}\right)^{\alpha_j} +\mathcal{O}\left(\frac{1}{n}\right) , \label{l8}
\end{align}
 as \( n\to\infty\).
Next we turn to the multiple cosine sums.

\begin{lemma}\label{lemma3} Let us consider fixed integers \(m\in\N\), \(m\leq j \leq 2m\), and let us define \(\mathcal{A}_j\) as the set of all tuples \( (\alpha_1 , \ldots, \alpha_j) \in \{1,\ldots, 2m\}^j\) that solve the equation \(\alpha_1 + \ldots + \alpha_j = 2m\), with the property that exactly \(2m-j\) of the \(\alpha_\nu\) take the value 2, and the remaining \(2(j-m)\) of the \(\alpha_\nu\) take the value 1. Accordingly, let us denote by \(\mathcal{A}_j^c\) the set of tuples \( (\alpha_1 , \ldots, \alpha_j) \in \{1,\ldots, 2m\}^j\) that solve the equation \(\alpha_1 + \ldots + \alpha_j = 2m\), which are not contained in \(\mathcal{A}_j\). Then for \((\alpha_1, \ldots, \alpha_j) \in \mathcal{A}_j\) we have
\[\frac{1}{n^m} \sum_{\substack{i_1 \neq \ldots \neq i_j \\  i_1, \ldots, i_j \in \{1,\ldots, n\} } } \cos \left(\frac{2\pi i_1 }{n}\right)^{\alpha_1}\cdots  \cos \left(\frac{2\pi i_j }{n}\right)^{\alpha_j} = (-1)^{j-m} \frac{(2(j-m))!}{2^j (j-m)!}  \left(1+\mathcal{O} \left(\frac{1}{n}\right)\right), \]
and for \((\alpha_1, \ldots, \alpha_j) \in \mathcal{A}_j^c\) we have
\[\frac{1}{n^m} \sum_{\substack{i_1 \neq \ldots \neq i_j \\  i_1, \ldots, i_j \in \{1,\ldots, n\} } } \cos \left(\frac{2\pi i_1 }{n}\right)^{\alpha_1}\cdots  \cos \left(\frac{2\pi i_j }{n}\right)^{\alpha_j} =\mathcal{O} \left(\frac{1}{n}\right), \quad \text{as}\quad n\to\infty.\]
 
\end{lemma}
\begin{proof}[Proof of Lemma \ref{lemma3}] We first note that the number, say \(\nu\),  of entries in a tuple \( (\alpha_1 , \ldots, \alpha_j) \in \{1,\ldots, 2m\}^j\) that solves  \(\alpha_1 + \ldots + \alpha_j = 2m\) is at least \(2(j-m)\). This follows directly from 
\[2m = \alpha_1 + \ldots+ \alpha_j \geq \nu + 2(j - \nu). \]
Hence, the set \(\mathcal{A}_j\) consists of all solutions of the equation  \(\alpha_1 + \ldots + \alpha_j = 2m\) that contain the minimal number of unit entries, meaning exactly \(2(j-m)\) many. Moreover, we note that we have for integers \(\ell, n \in \mathbb{N}\)
\begin{equation} \label{l9} \sum_{\nu =1}^n \cos\left( \frac{2\pi \nu}{n}\right)^{\ell} = \begin{cases} \frac{n}{2^{\ell}} \binom{\ell}{\ell/2} & \text{if \(\ell\) is even,}  \\  0 & \text{if \(\ell\) is odd}.\end{cases}
\end{equation}
Let us initially focus on the case \(j=m\), in which the set  \(\mathcal{A}_j\) only consists of the tuple \((\alpha_1, \ldots, \alpha_m) = (2,\ldots, 2)\). Using similar arguments to those  used in the proof of Lemma \ref{lemma2}, and then using \eqref{l9}, we see that for  this tuple
\begin{align}\nonumber \sum_{\substack{i_1 \neq \ldots \neq i_m \\  i_1, \ldots, i_m \in \{1,\ldots, n\} } } &\cos \left(\frac{2\pi i_1 }{n}\right)^{2}\cdots  \cos \left(\frac{2\pi i_m }{n}\right)^{2}\\ \nonumber
& = \sum_{ i_1, \ldots, i_m \in \{1,\ldots, n\}  } \cos \left(\frac{2\pi i_1 }{n}\right)^{2}\cdots  \cos \left(\frac{2\pi i_m }{n}\right)^{2} + \mathcal{O}\left(n^{m-1}\right)\\ \nonumber
&= \left(\sum_{\nu =1}^n \cos\left( \frac{2\pi \nu}{n}\right)^{2}\right)^m  + \mathcal{O}\left(n^{m-1}\right)\\ \label{l10}
	&=  \left(\frac{n}{2}\right)^m+ \mathcal{O}\left(n^{m-1}\right), \quad \text{as}\quad n\to\infty. 
\end{align}
If, on the other hand, we consider a tuple \((\alpha_1, \ldots, \alpha_m)  \in \mathcal{A}_m^c\), then there is at least one entry equal to one. Without loss of generality, we may assume that \(\alpha_m =1\).  Then, using  \eqref{l9}, we have
\begin{align} \label{l11}&\sum_{\substack{i_1 \neq \ldots \neq i_m \\  i_1, \ldots, i_m \in \{1,\ldots, n\} } } \cos \left(\frac{2\pi i_1 }{n}\right)^{\alpha_1}\cdots  \cos \left(\frac{2\pi i_m }{n}\right)^{\alpha_m =1}\\\nonumber
&\quad = \sum_{\substack{i_1 \neq \ldots \neq i_{m-1} \\  i_1, \ldots, i_{m-1} \in \{1,\ldots, n\} } } \cos \left(\frac{2\pi i_1 }{n}\right)^{\alpha_1}\cdots  \cos \left(\frac{2\pi i_{m-1} }{n}\right)^{\alpha_{m-1}}\sum_{\substack{i_{m} \in\{1,\ldots, n\} \\  i_m \notin  \{i_1, \ldots, i_{m-1}\} } } \cos \left(\frac{2\pi i_m }{n}\right)\\\nonumber
&\quad =  - \sum_{\substack{i_1 \neq \ldots \neq i_{m-1} \\  i_1, \ldots, i_{m-1} \in \{1,\ldots, n\} } } \cos \left(\frac{2\pi i_1 }{n}\right)^{\alpha_1}\cdots  \cos \left(\frac{2\pi i_{m-1} }{n}\right)^{\alpha_{m-1}}\sum_{  i_m \in  \{i_1, \ldots, i_{m-1}\} }  \cos \left(\frac{2\pi i_m }{n}\right)\\\nonumber
&\quad = - \sum_{\nu =1}^{m-1}  \sum_{\substack{i_1 \neq \ldots \neq i_{m-1} \\  i_1, \ldots, i_{m-1} \in \{1,\ldots, n\} } } \cos \left(\frac{2\pi i_1 }{n}\right)^{\alpha_1}\cdots   \cos \left(\frac{2\pi i_{\nu} }{n}\right)^{\alpha_{\nu}+1} \cdots   \cos \left(\frac{2\pi i_{m-1} }{n}\right)^{\alpha_{m-1}}, 
 \end{align}
and each of the latter sums is of order \(\mathcal{O}\left(n^{m-1}\right)\), as \(n\to\infty\), which proves the statement of the lemma in the case \(j=m\). 

Next we deal with the case \(m < j \leq 2m \).  Let us first consider a tuple  \((\alpha_1, \ldots, \alpha_j)  \in \mathcal{A}_j^c\), so the tuple consists of at least  \(2(j-m) +1\) many entries taking the value one. Without loss of generality, we may assume that  \(\alpha_j = 1 \). Then, by the same reasoning as in \eqref{l11}, we have

\begin{align*} &\sum_{\substack{i_1 \neq \ldots \neq i_j \\  i_1, \ldots, i_j \in \{1,\ldots, n\} } } \cos \left(\frac{2\pi i_1 }{n}\right)^{\alpha_1}\cdots  \cos \left(\frac{2\pi i_j }{n}\right)^{\alpha_j =1}\\
&\quad = - \sum_{\nu =1}^{j-1}  \sum_{\substack{i_1 \neq \ldots \neq i_{j-1} \\  i_1, \ldots, i_{j-1} \in \{1,\ldots, n\} } } \cos \left(\frac{2\pi i_1 }{n}\right)^{\alpha_1}\cdots   \cos \left(\frac{2\pi i_{\nu} }{n}\right)^{\alpha_{\nu}+1} \cdots   \cos \left(\frac{2\pi i_{j-1} }{n}\right)^{\alpha_{j-1}},
 \end{align*}
where the number of entries equal to one in the tuple of exponents of the cosine terms in each of the sums 
\[ \sum_{\substack{i_1 \neq \ldots \neq i_{j-1} \\  i_1, \ldots, i_{j-1} \in \{1,\ldots, n\} } } \cos \left(\frac{2\pi i_1 }{n}\right)^{\alpha_1}\cdots   \cos \left(\frac{2\pi i_{\nu} }{n}\right)^{\alpha_{\nu}+1} \cdots   \cos \left(\frac{2\pi i_{j-1} }{n}\right)^{\alpha_{j-1}}, \quad \nu =1, \ldots, j-1,\]
is reduced at least by one and at most by two, compared to the inital tuple \((\alpha_1, \ldots, \alpha_j)\).  We can repeat this procedure of reduction \(j-m\) times, and in this process the number of entries equal to one in the tuples of exponents of the resulting sums are reduced by at most \(2(j-m)\). Hence, each of the produced sums is of the form \eqref{l11}, which means it is of order \(\mathcal{O}\left(n^{m-1}\right)\). This shows that for \((\alpha_1, \ldots, \alpha_j)  \in \mathcal{A}_j^c\)
\[\sum_{\substack{i_1 \neq \ldots \neq i_j \\  i_1, \ldots, i_j \in \{1,\ldots, n\} } } \cos \left(\frac{2\pi i_1 }{n}\right)^{\alpha_1}\cdots  \cos \left(\frac{2\pi i_j }{n}\right)^{\alpha_j } = \mathcal{O}\left(n^{m-1}\right), \quad \text{as}\quad n\to\infty.\]
Finally, let us consider a tuple  \((\alpha_1, \ldots, \alpha_j)  \in \mathcal{A}_j\), so the tuple consists of exactly  \(2(j-m)\) many entries taking the value one, and the remaining \(2m-j\) entries take the value two. Without loss of generality, we may assume that \(\alpha_1 = \ldots =\alpha_{2m-j}=2\) and \(\alpha_{2m-j +1}=  \ldots =\alpha_j = 1 \). Then, again using \eqref{l9}, we have
\begin{align*}&\sum_{\substack{i_1 \neq \ldots \neq i_j \\  i_1, \ldots, i_j \in \{1,\ldots, n\} } } \cos \left(\frac{2\pi i_1 }{n}\right)^{2}\cdots \cos \left(\frac{2\pi i_{2m-j} }{n}\right)^{2}  \cos \left(\frac{2\pi i_{2m-j+1} }{n}\right)^{} \cdots\cos \left(\frac{2\pi i_j }{n}\right)^{ }\\
&= -\sum_{\nu=1}^{2m-j}\sum_{\substack{i_1 \neq \ldots \neq i_{j-1} \\  i_1, \ldots, i_{j-1} \in \{1,\ldots, n\} } } \cos \left(\frac{2\pi i_\nu }{n}\right)^{3} \Bigg(\prod_{\substack{\ell =1\\ \ell \neq \nu}}^{2m-j} \cos \left(\frac{2\pi i_\ell }{n}\right)^{2} \Bigg) \Bigg(\prod_{\ell =2m-j+1}^{j-1}  \cos \left(\frac{2\pi i_\ell }{n}\right)^{}\Bigg)\\
& -\sum_{\nu=2m-j+1}^{j-1}\sum_{\substack{i_1 \neq \ldots \neq i_{j-1} \\  i_1, \ldots, i_{j-1} \in \{1,\ldots, n\} } } \Bigg(\prod_{\ell =1}^{2m-j} \cos \left(\frac{2\pi i_\ell }{n}\right)^{2} \Bigg)   \cos \left(\frac{2\pi i_\nu }{n}\right)^{2} \Bigg(\prod_{\substack{\ell =2m-j+1\\ \ell \neq \nu}}^{j-1}  \cos \left(\frac{2\pi i_\ell }{n}\right)^{}\Bigg).
\end{align*}
We observe that, regarding in the first sum, each multiple sum indexed by \(\nu = 1, \ldots, 2m-j\) 
\[\sum_{\substack{i_1 \neq \ldots \neq i_{j-1} \\  i_1, \ldots, i_{j-1} \in \{1,\ldots, n\} } } \cos \left(\frac{2\pi i_\nu }{n}\right)^{3} \Bigg(\prod_{\substack{\ell =1\\ \ell \neq \nu}}^{2m-j} \cos \left(\frac{2\pi i_\ell }{n}\right)^{2} \Bigg) \Bigg(\prod_{\ell =2m-j+1}^{j-1}  \cos \left(\frac{2\pi i_\ell }{n}\right)^{}\Bigg)\]
is of dimension \(j-1\), and among the exponents of the cosine terms, we find exactly \(2j-2m-1\) of them being equal to one. Hence, as before by using \eqref{l9}, we can reduce such a sum in \(j-1-m\) steps to an \(m\)-dimensional sum of the form \eqref{l11}, meaning that each of these sums is of order \(\mathcal{O}\left(n^{m-1}\right)\). Moreover, we note that the sums indexed by \(\nu = 2m-j+1, \ldots, j-1\) are actually independent of the index \(\nu\). This gives
\begin{align*}&\sum_{\substack{i_1 \neq \ldots \neq i_j \\  i_1, \ldots, i_j \in \{1,\ldots, n\} } } \cos \left(\frac{2\pi i_1 }{n}\right)^{2}\cdots \cos \left(\frac{2\pi i_{2m-j} }{n}\right)^{2}  \cos \left(\frac{2\pi i_{2m-j+1} }{n}\right)^{} \cdots \cos \left(\frac{2\pi i_j }{n}\right)^{ }\\
& = -(2(j-m)-1)\sum_{\substack{i_1 \neq \ldots \neq i_{j-1} \\  i_1, \ldots, i_{j-1} \in \{1,\ldots, n\} } } \Bigg(\prod_{\ell =1}^{2m-j +1} \cos \left(\frac{2\pi i_\ell }{n}\right)^{2} \Bigg)   \Bigg(\prod_{\ell =2m-j+1}^{j-1}  \cos \left(\frac{2\pi i_\ell }{n}\right)^{}\Bigg)\\
& \quad + \mathcal{O}\left(n^{m-1}\right), \quad \text{as}\quad n\to\infty.
\end{align*}
We can repeat this process with the latter sum, and after \(j-m\) steps we arrive at
 \begin{align*}&\sum_{\substack{i_1 \neq \ldots \neq i_j \\  i_1, \ldots, i_j \in \{1,\ldots, n\} } } \cos \left(\frac{2\pi i_1 }{n}\right)^{2}\cdots \cos \left(\frac{2\pi i_{2m-j} }{n}\right)^{2}  \cos \left(\frac{2\pi i_{2m-j+1} }{n}\right)^{} \cdots\cos \left(\frac{2\pi i_j }{n}\right)^{ }\\
&= (-1)^{j-m} \prod_{\nu =1}^{j-m} (2\nu -1) \sum_{\substack{i_1 \neq \ldots \neq i_m \\  i_1, \ldots, i_m \in \{1,\ldots, n\} } } \cos \left(\frac{2\pi i_1 }{n}\right)^{2}\cdots  \cos \left(\frac{2\pi i_m }{n}\right)^{2}  + \mathcal{O}\left(n^{m-1}\right) \\
&=  (-1)^{j-m} \frac{(2(j-m))!}{(j-m)! 2^j} n^m+ \mathcal{O}\left(n^{m-1}\right),
\end{align*} 
as \(n \to \infty\), where we used \eqref{l10} in the last step.
\end{proof}
\bigskip
We now can return to the proof of Lemma \ref{lemma1}. Applying Lemma \ref{lemma3} to \eqref{l8} yields
\begin{align}\label{l12}  \int_{\S_n} X_n^{2m} \d\sigma &=\sum_{j=m}^{2m}  \frac{(-1)^{j-m}}{j!} \frac{(2(j-m))! (2m)! }{2^j (j-m)!}    \sum_{(\alpha_1 , \ldots, \alpha_j) \in \mathcal{A}_j} \frac{1}{(\alpha_1 +1)! \cdots (\alpha_j+1)!}\\ \nonumber
&\quad\quad  +\mathcal{O}\left(\frac{1}{n}\right) , \quad n\to\infty.
\end{align}
Using the definition of the sets \(\mathcal{A}_j\) in Lemma \ref{lemma3}, we can evaluate the inner sum in \eqref{l12} explicitly
\begin{align*}\sum_{(\alpha_1 , \ldots, \alpha_j) \in \mathcal{A}_j} \frac{1}{(\alpha_1 +1)! \cdots (\alpha_j+1)!} &= \frac{1}{(2!)^{2(j-m)}} \frac{1}{(3!)^{2m-j}}\binom{j}{2(j-m)} \\
&=\frac{1}{2^j 3^{2m-j}} \frac{j!}{(2(j-m))! (2m-j)!}.
\end{align*}
Substituting this into \eqref{l12}, some cancellation and shifting the summation index then leads to
\begin{align*}  \int_{\S_n} X_n^{2m} \d\sigma &= \frac{(2m)!}{4^m m!} \sum_{j=0}^m \binom{m}{j} \left(-\frac{1}{4}\right)^j \left(\frac{1}{3}\right)^{m-j} +\mathcal{O}\left(\frac{1}{n}\right)\\
&= \frac{(2m)!}{m!} \left(\frac{1}{48}\right)^{m}+\mathcal{O}\left(\frac{1}{n}\right), \quad \text{as}\quad  n\to\infty.
\end{align*}
This proves the statement \eqref{l1} of Lemma \ref{lemma1} for even moments. In order to verify \eqref{l2}, we observe for \(n\geq 2m\) that the analogue of the representation \eqref{l4} for odd moments is 
\begin{align*}\nonumber  \int_{\S_n} X_n^{2m-1} \d\sigma 
&=  \frac{1}{n^{3m-3/2}}\sum_{j=1}^{2m-1}  \frac{1}{j!}  \frac{(n-j)!}{n!}   \sum_{\substack{\alpha_1 , \ldots, \alpha_j \in \{1,\dots, 2m-1\} \\ \alpha_1 + \ldots + \alpha_j = 2m-1 }} \frac{(2m-1)!}{\alpha_1! \cdots \alpha_j!}   \\
&\quad \Bigg(\sum_{\substack{k_1 \neq \ldots \neq k_j \\  k_1, \ldots, k_j \in \{1,\ldots, n\}  } } k_1^{\alpha_1} \cdots k_j^{\alpha_j} \Bigg)  \sum_{\substack{i_1 \neq \ldots \neq i_j \\  i_1, \ldots, i_j \in \{1,\ldots, n\} } } \cos \left(\frac{2\pi i_1 }{n}\right)^{\alpha_1}\cdots  \cos \left(\frac{2\pi i_j }{n}\right)^{\alpha_j}. 
\end{align*}
We are going to show inductively for every \(j=1,\ldots, 2m-1\) the following claim:  for every tuple \((\alpha_1 , \ldots, \alpha_j) \in \{1,\dots, 2m-1\}^j\) satisfying \(\alpha_1 + \ldots + \alpha_j =2m-1\) we have 
\begin{equation}\label{l13}\sum_{\substack{i_1 \neq \ldots \neq i_j \\  i_1, \ldots, i_j \in \{1,\ldots, n\} } } \cos \left(\frac{2\pi i_1 }{n}\right)^{\alpha_1}\cdots  \cos \left(\frac{2\pi i_j }{n}\right)^{\alpha_j} =0.
\end{equation}
The case \(j=1\) follows directly from \eqref{l9}. Let us assume the claim holds true for \(j-1\) for some \(j>1\). As at least one of the \(\alpha_1, \ldots, \alpha_j\) must be odd, without loss of generality, we can assume that \(\alpha_j\) is odd. Then, using \eqref{l9}, we have
\begin{align*} &\sum_{\substack{i_1 \neq \ldots \neq i_j \\  i_1, \ldots, i_j \in \{1,\ldots, n\} } } \cos \left(\frac{2\pi i_1 }{n}\right)^{\alpha_1}\cdots  \cos \left(\frac{2\pi i_j }{n}\right)^{\alpha_j }\\
&\quad = - \sum_{\nu =1}^{j-1}  \sum_{\substack{i_1 \neq \ldots \neq i_{j-1} \\  i_1, \ldots, i_{j-1} \in \{1,\ldots, n\} } } \cos \left(\frac{2\pi i_1 }{n}\right)^{\alpha_1}\cdots   \cos \left(\frac{2\pi i_{\nu} }{n}\right)^{\alpha_{\nu}+\alpha_j} \cdots   \cos \left(\frac{2\pi i_{j-1} }{n}\right)^{\alpha_{j-1}},
 \end{align*}
and each of the sums vanishes due to the induction hypothesis. This proves the claim \eqref{l13}, thus statement \eqref{l2}, and thereby completes the proof of Lemma \ref{lemma1}.

\end{proof}

\section*{Acknowledgements}
The author would like to thank Ulrich Abel for pointing out the interesting questions inspiring this article.


\begin{thebibliography}{00}

\bibitem[1]{Abel} U. Abel, Coincidence of the barycentre and the geometric centre of weighted points, The Mathematical Gazette, 2019; 103(558): 409--415.
\bibitem[2]{Bill} P. Billingsley, Probability and Measure, 3rd edition, Wiley, 1995.
\bibitem[3] {Knuth} D. Knuth, Johann Faulhaber and the Sums of Powers, Mathematics of Computation, 1993; 61 (203), 277--294.
\bibitem[4]{Woeg}  G. J. Woeginger, Nothing new about equiangular polygons, Amer. Math. Monthly 120 (2013) pp. 849--850.

\end{thebibliography}
\end{document}